\documentclass{amsart}
\usepackage[T1]{fontenc}
\usepackage{amsmath, amsthm, amsfonts, amssymb}
\usepackage{enumerate}

\theoremstyle{definition}

\newtheorem{defn}{Definition}[section]

\theoremstyle{remark}
\newtheorem{rmk}[defn]{Remark}

\theoremstyle{plain}

\newtheorem{cor}[defn]{Corollary}

\newtheorem{lem}[defn]{Lemma}

\newtheorem{thm}[defn]{Theorem}

\numberwithin{equation}{section}

\newcommand{\DF}{\mathcal{E}}
\newcommand{\domDF}{\mathcal{F}}

\newcommand{\pl}{Y} 

\DeclareMathSizes{10}{9}{7}{5}

\title[Harmonic Sierpinski Gaskets]{Measurable Riemannian structure on higher dimensional harmonic Sierpinski gaskets}

\author[Chari, Frisch, Kelleher, Rogers]{Sara Chari, Joshua Frisch, Daniel~J. Kelleher, Luke~G.~Rogers}
\thanks{Authors supported in part by the National Science Foundation through grant DMS-1262929.}

\begin{document}

\begin{abstract}
We prove existence of a measurable Riemannian structure on higher-dimensional harmonic Sierpinski gasket fractals and deduce Gaussian heat kernel bounds in the geodesic metric.  Our proof differs from that given by Kigami for the usual Sierpinski gasket~\cite{KigMsbleRGeom} in that we show the geodesics are de~Rham curves, for which there is an extensive regularity theory.
\end{abstract}
%

\maketitle
%

%

\section{Introduction and Main Result} \label{introduction}

The basic elements of analysis on Sierpinski-Gasket type fractals were developed, in a more general context, by Kusuoka~\cite{Ku89} and Kigami~\cite{Kig89}.   The following definitions are from~\cite{KigMsbleRGeom}, though our presentation is different at some points and was influenced by~\cite{TepHarmCoords,Tep04}; detailed proofs of results not demonstrated here may be found in~\cite{Kigbook}.

\begin{defn}
For $N\geq2$ the classical $N$-Sierpinski Gasket $K_N$ is defined as follows. Let $\{p_j\}_{j=1}^N$ be the vertices of a regular $N$ simplex in $\mathbb{R}^{N-1}$ such that $|p_j-p_k|=1$ if $j\neq k$, and $F_i:\mathbb{R}^{N-1}\to\mathbb{R}^{N-1}$ be $F_j(x)=(x-p_j)/2 +p_j$. Then $K_N$ is the unique non-empty compact set such that $K_N=\cup_1^N F_j(K_N)$.  Note that $K_2$ is an interval and $K_3$ is the usual Sierpinski Gasket.
\end{defn}

Fix $N\geq2$. For notational simplicity we write $K$ for $K_N$.  Then $K$ is post-critically finite, with post-critical set $V_0=\{p_1,\dotsc,p_N\}$. Let $W_m=\{1,\dotsc,N\}^m$ denote the set of words length $|w|=m$, so $w\in W_m$ is $w=w_1\dotsm w_m$ with each $w_j\in\{1,\dotsc,N\}$. Let $W_*=\cup_m W_m$ and for $w=w_1\dotsm w_m\in W_*$ define $F_w=F_{w_1}\circ\dotsm\circ F_{w_m}$.  We set $V_m=\cup_{w\in W_m} F_w(V_0)$, which we call the set of scale~$m$ vertices; evidently this is the level~$m$ critical set. We write $p\sim_m q$ if $p$ and $q$ are both in $F_w(V_0)$ for some $w\in W_m$.

For each $N\geq2$ there is a non-negative definite, symmetric, quadratic form  on $K=K_N$ which may be defined as a limit of forms on the sets $V_m$.
\begin{defn}
For $u,v$ continuous functions on $K$ let
\begin{equation}\label{eqn:defnofDFm}
	\DF_m(u,v) = \Bigl(\frac{N+2}{N} \Bigr)^m \sum_{p\sim_m q} \bigl( u(p)-u(q)\bigr)\bigl( v(p)-v(q) \bigr).
	\end{equation}
Write $\DF_m(u)=\DF_m(u,u)$.  It is known that $\DF_m(u)\leq\DF_{m+1}(u)$ for all $m\geq0$, so that $\lim_m\DF_m(u)$ exists. Define $\domDF=\{u:\lim_m\DF_m(u)<\infty\}$ and $\DF(u,v)=\lim_m \DF_m(u,v)$. When $N=2$, $\domDF$ is the space of functions with one derivative in $L^2$ that vanishes at the endpoints, and $\DF$ is the $L^2$ norm of the derivative.
\end{defn}

A proof of the following result is in~\cite{Kigbook}.
\begin{thm}
Let $\mu$ be an atomless Borel probability measure on $K$ with $\mu(O)>0$ if $O$ is non-empty and open. Then $(\DF,\domDF)$ is a  local regular Dirichlet form on $L^2(K,\mu)$.
\end{thm}

The analytic structure corresponding to the Dirichlet form $(\DF,\domDF)$ is greatly elucidated by considering the harmonic gasket, or gasket in harmonic coordinates.  A function $h\in\domDF$ is harmonic if $\DF_m(h)=\DF_{m+1}(h)$ for all $m$. Since this involves minimization of a quadratic functional the minimizer is given by a linear operator on the values on $V_m$.  When $m=0$ we write $H_j$ for the operator taking the values on $V_0$ to those on $F_j(V_0)$, more precisely:
\begin{equation*}
	H_j: \bigl(h(p_1),\dotsc,h(p_N)\bigr) \mapsto \bigl(h(F_j p_1),\dotsc,h(F_j p_N)\bigr).
	\end{equation*}
By self-similarity we find $H_j: \bigl(h(F_w p_1),\dotsc,h(F_w p_N)\bigr) \mapsto \bigl(h(F_wj p_1),\dotsc,h(F_wj p_N)\bigr)$ for any $w\in W_*$, and therefore the values of $h$ on $V_*$ are determined by composing the $H_j$ (or multiplying the corresponding matrices, which we also denote by $H_j$); this determines $h$ on $K$ by continuity, so the fact that the $H_j$ are invertible ensures the harmonic functions are in one-to-one correspondence with functions on $V_0$ and are an $N$-dimensional space.  The latter may be seen by explicitly computing (from~\eqref{eqn:defnofDFm}) that
\begin{equation}
	H_1= \frac{1}{N+2} \begin{bmatrix}
		N+2	& 0  \\
		2	& I+J 
		\end{bmatrix}
	\end{equation}
where $I_{N-1}$ is the identity, $J_{N-1}$ is the size $N-1$ square matrix with all entries equal $1$, and we have written $0$ and $2$ for the length  $(N-1)$ vectors with all entries equal $0$ and $2$ respectively.  Symmetry implies the other $H_j$ may be obtained by cyclic row and column permutations.

\begin{defn}
Fix $N\geq2$ and for $1\leq j\leq N$ let $\psi_j$ denote the harmonic function on $K_N$ which is $1$ at $p_j$ and $0$ on $V_0\setminus\{p_j\}$.  The function $\Psi:K\to\mathbb{R}^N$ given by $\Psi=\bigl((\psi_1,\dotsc,\psi_N)-(1,\dotsc,1)/N\bigr)/\sqrt{2}$ is injective (see~\cite{Kig93}), so it is a homeomorphism onto its image $X_N$, which we call the {\em harmonic N-Sierpinski gasket}.
\end{defn}

Constant functions are harmonic, so $\sum_j \psi_j\equiv1$ and the harmonic N-Sierpinski gasket $X_N$ lies in the subspace $\pl=\bigl\{(x_1,\dotsc,x_N):\sum_j x_j=0\bigr\}$, which we identify with $\mathbb{R}^{N-1}$.  Moreover $\Psi(V_0)$ is the set of vertices of a regular $N$-simplex with unit length sides in this subspace, so we identify $\Psi(V_0)$ with $V_0$ via $\Psi(p_j)=p_j$.   It is an important fact that $X_N$ is a self-affine set under maps conjugate to the $F_j$ via $\Psi$.

\begin{thm}[\protect{\cite{Kig93}}]
For each $1\leq j\leq N$ let $T_j$ be the linear map on $\pl$ that contracts the direction $\Psi(p_j)$ by the factor $N/(N+2)$ and all orthogonal directions by the factor $1/(N+2)$.  Let $S_j(x)=T_j(x-p_j)+p_j$, so $S_j:Y\to Y$.  Then $\Psi\circ F_j= S_j\circ\Psi$  for each $j$ and therefore $X_N=\cup_1^N S_j(X_N)$ is self-affine.
\end{thm}

\begin{rmk}
We will sometimes write $S_j$ in the equivalent form
\begin{equation}\label{eq:Sj}
	S_j x = \frac{x}{N+2} + \frac{1}{N+2} \Bigl( 2+ \frac{(N-1)x\cdot p_j}{\|p_j\|^2} \Bigr) p_j
	\end{equation}
\end{rmk}

It is a special case of results of Kusuoka~\cite{Ku89} that there is a Carr\'{e} du Champs measure $\nu$ for the Dirichlet form $(\DF,\domDF)$ and an associated $\nu$-a.e.\ defined metric $Z$ that can be expressed in terms of the operators $T_j$ as in the following results.  Recall that the Hilbert-Schmidt norm on a linear operator on $\mathbb{R}^{N-1}$ may be defined by setting $\|T\|_{\text{HS}}^2$ to be the sum of the squares of the coefficients of the associated matrix with respect to the standard orthonormal basis.

\begin{thm}[\protect{\cite{Ku89}}]
For $w=w_1\dotsm w_m\in W_*$ let $T_w=T_{w_1}\dotsm T_{w_m}$ and $Z_m(w) = T_w T_w^t/\|T_w\|_{\text{HS}}^2$, where the adjoint is taken with respect to the usual inner product on $\pl$. Let $\Sigma=\{1,\dotsc,N\}^\mathbb{N}$ denote the space of infinite words.
\begin{enumerate}[(i)]
\item There is a unique, atomless, Borel regular probability measure $\nu$ on $\Sigma$, called the Kusuoka measure, such that for any $w\in W_*$
\begin{equation*}
	\nu( w\Sigma)=\frac{1}{N-1} \Bigl( \frac{N+2}{N}\Bigr)^{|w|} \bigl\| T_w \bigr\|^2_{\text{HS}}.
	\end{equation*}
\item The limit $Z(w)=\lim_{m\to\infty} Z_m(w_1\dotsm z_m)$ exists  and is the orthogonal projection onto its image for $\nu$-a.e.\ $w=w_1 w_2\dotsm\in\Sigma$.
\end{enumerate}
\end{thm}

Both $\nu$ and $Z$ can be transferred to $K$ using the obvious projection $\pi:\Sigma\to K$, which is defined by $\pi(w)=\lim_m F_{w_1\dotsm w_m}K$,  because this projection is injective off the measure-zero set $W_*$.  For this reason we abuse notation to use $\nu$ and $Z$ for the measure and a.e.-defined linear operator obtained by pushing forward under $\pi$.  

Together, the form $(\DF,\domDF)$, measure $\nu$ and metric  $Z$ are analogues of the Riemannian energy, volume and metric in that if $\mathcal{C}$ denotes the $\nu$-measurable functions $K\to\pl$ then the following theorem of Kusuoka  holds.
\begin{thm}[\protect{\cite{Ku89}}]
There is $\tilde{\nabla}:\domDF\to\mathcal{C}$ such that for all $u,v\in\domDF$
\begin{equation*}
	\DF(u,v) = \int_K \langle \tilde{\nabla}u,Z\tilde{\nabla}v\rangle\,d\nu.
\end{equation*}
\end{thm}

The sense in which this structure is simplified by considering the harmonic gasket $X_N$ is captured by the following theorem which relates $\tilde{\nabla}$ to the classical gradient $\nabla$ on $\pl=\mathbb{R}^{N-1}$.  Roughly speaking, it says that the analytic structure we have on $K$ is just the restriction of the usual smooth structure on $\pl$ to the harmonic gasket $X$, transferred to $K$ via $\Psi$, when $K$ is endowed with the Kusuoka measure.
\begin{thm}[\protect{\cite{Ku89,KigMsbleRGeom}}]
Let $\mathcal{D}=\{u=v|_X\circ\Psi: v \text{ is } C^1 \text{ on a neighborhood of $X$ in $\pl$}\}$.  Then $\mathcal{D}$ is dense in $\domDF$ with respect to the norm $\|u\|^2=\DF(u)+\|u\|_\infty^2$.  Moreover for $u,v\in\mathcal{D}$ we have $\tilde{\nabla}u=Z\nabla u$ $\nu$-a.e.\ and
\begin{equation*}
	\DF(u,v) = \int_K \langle \nabla u,Z\nabla v\rangle\, d\nu.
	\end{equation*}
\end{thm}

One main result of~\cite{KigMsbleRGeom}, see also Teplyaev~\cite{Tep04}, was that in addition to the energy, measure and metric structure described above, the harmonic Sierpinski gasket $X_3$ admits a geodesic distance analogous to the Riemannian distance. Moreover the distance between two points may be computed by integrating the norm, computed with respect to the metric $Z$, of the directional derivative along a geodesic path joining these points.  Kigami states in~\cite{KigMsbleRGeom}, but does not verify, that similar results can be proved for $X_N$, $N>3$ using a similar but more complicated argument.  The purpose of the present work is to prove this claim by a slightly different approach involving de~Rham curves.  Specifically we prove the following.

\begin{thm}\label{thm:geodesicspace}
If $p,q\in X$ there is a Euclidean geodesic from $p$ to $q$ in $X$.  There is a $C^1$ function $g_{pq}:[0,1]\to X$ that parametrizes this curve, and the length of the curve is
\begin{equation}\label{eqn:geodesiceqn}
\int_0^1 \left\langle g_{pq}'(t), Z\bigl(g_{pq}(t)\bigr)\, g_{pq}'(t) \right\rangle^{1/2}\, dt.
\end{equation}
\end{thm}

Knowing Theorem~\ref{thm:geodesicspace} it follows from the work in Section~6 of Kigami's paper~\cite{KigMsbleRGeom} that the heat semigroup associated to the Dirichlet form $\DF$ on $L^2(\nu)$ has a jointly continuous kernel $p(t,x,y)$ which has Gaussian bounds with respect to the geodesic distance $d_\ast$ on $X$.  Precisely, we obtain the following.
\begin{cor} There are constants $c_1,c_2,c_3,c_4$ such that if $B(x,r)$ denotes the ball of radius $r$ around $x$ in the metric $d_\ast$ then for $t\in(0,1]$ and $x\in X$
\begin{equation*}
\frac{c_1}{\nu(B(x,\sqrt{t}))} \exp \biggl( \frac{-c_2 d_\ast(x,y)^2}{t} \biggr)
\leq p(t,x,y)
\leq \frac{c_3}{\nu(B(x,\sqrt{t}))} \exp \biggl( \frac{-c_4 d_\ast(x,y)^2}{t} \biggr)
\end{equation*}
\end{cor}

Note that Kajino~\cite{Kajino} has proved more refined estimates for the heat kernel using a related distance estimate on the classical Sierpinski gasket $X_3$.  His approach is not applicable to the gaskets with $N>3$, as he points out in~\cite{Kajino2}, and it is not known whether these more precise results are true in this setting.

\section{Geodesics on $X_N$: Proof of Main Result}

Fix $N\geq2$ and write $X=X_N$ for the harmonic N-Sierpinski gasket.  Recall that there are affine maps $S_j:\pl\to\pl$ so that $X=\cup_1^N S_j(X)$.  To obtain our geodesics we consider the subsets of $X$ that are self-affine under two of the maps $S_j$.  By symmetry it is sufficient to consider the unique, non-empty, compact set $\Gamma$ such that $\Gamma=S_1(\Gamma)\cup S_2(\Gamma)$.  Our initial goal is to prove the following.

\begin{thm}\label{thm:Gammarectifiable}
$\Gamma$ is a $C^1$ plane curve connecting  $p_1,p_2\in V_0$.  If $N\geq3$ it is $C^1$ but not $C^2$; in fact it has a parametrization in which the derivative is H\"older continuous with H\"older exponent
\begin{equation*}
\frac{\log N +2\log 2 - 2\log (1+\sqrt{4N+1}) }{\log (1+\sqrt{4N+1}) - \log 2 -\log(N+2)   }.
\end{equation*}
Moreover, if $x\in\Gamma$, so $x=S_w(X)$  for some infinite word $w$ with letters from $\{1,2\}$ then $Z(w)$ is projection onto the tangent direction of $\Gamma$ at $w$.
\end{thm}

We give several intermediate results before proving the theorem.  Let $P$ be the $2$-dimensional subspace of $\pl$ that contains  $p_1$ and $p_2$ and let $\Pi$ denote the orthogonal projection of $\pl$ onto $P$.

\begin{lem}\label{lem:GammainP}
$P$ is invariant under $S_1$ and $S_2$ and therefore contains $\Gamma$.  Moreover $\Pi(V_0\setminus\{p_1,p_2\})$ is a single point $p$.  
\end{lem}
\begin{proof}
It suffices by symmetry to consider the action of $S_1$.  If $x\in P$ we may orthogonally decompose it as $x=\alpha p_1+ (x-\alpha p_1)$.  Then \begin{equation*}
	S_1(x)=\frac{N}{N+2}\alpha p_1  + \frac{1}{N+2} (x-\alpha p_1)
	= \frac{N-1}{N+2} \alpha p_1 + \frac{1}{N+2} x
	\end{equation*}
which is a linear combination of $p_1\in P$ and $x\in P$, so lies in $P$.  This proves invariance of $P$ under $S_1$ and $S_2$.   It follows that the closure of $\{S_w p_1: w\in W_*\}$ is a subset of $P$, but since it is also non-empty, compact, and equal to the union of its images under $S_1$ and $S_2$ it must be equal to $\Gamma$, proving $\Gamma\subset P$.  Finally, symmetry ensures that all points in $V_0\setminus\{p_1,p_2\}$ project to the same point $p\in P$, and that $|p_1-p|=|p_2-p|$. 
\end{proof}

Many other properties of $\Gamma$ now follow from the fact that $\Gamma$ is an arc on a classical type of curve  introduced by de~Rham~\cite{deRham1,deRham2}. This fact was noted by Teplyaev in~\cite{Tep04} for the case of the classical Sierpinski gasket, where he also gave a number of results about the energy and Laplacian in harmonic coordinates.

\begin{defn}\label{def:deRhamcurve}
Fix a polygonal arc $A_0$ and a ratio $r\in(0,1/2)$. Inductively suppose we have defined the polygon $A_{k-1}$, introduce two new vertices on each side of $A_{k-1}$ so as to divide the side into subintervals with length ratios $r:1-2r:r$ and let $A_k$ be the polygon defined by the new vertices ordered as they were on $A_{k-1}$.  De Rham proved the $A_k$ converge to a continuous curve $A$; we call $A$ a de~Rham curve.
\end{defn}

To understand the de~Rham curves it helps to recognize that the midpoints of segments of the initial polygon are unchanged by the construction, and that  the construction of the arc between two adjacent midpoints is independent of the remainder of the curve, so that it is sufficient to consider the region between two such midpoints.  We can therefore perform the construction with an initial polygon having three vertices and two edges to obtain a de~Rham curve between the midpoints of these edges.  The significance for our problem is in the following theorem.

\begin{thm}\label{thm:GammaisdeRham}
$\Gamma$ is an arc of a de~Rham curve in $P$ with scale factor $r=\frac{1}{N+2}$.
\end{thm}
\begin{proof}
Let  $A_0$ be the line segments from $2p_1$ to $0$ and $0$ to $2p_2$, so that $p_1$ and $p_2$ are the midpoints of the sides.  Let $A_k$, $A$ be as in Definition~\ref{def:deRhamcurve} with ratio $r$, except that we retain only the arc between $p_1$ and $p_2$.

Write $[a,b]$ for the line segment from $a$ to $b$ and $D$ for one step of the de~Rham construction.    A key observation is that $D$ commutes with $S_1$ and $S_2$ because $D$ depends only on dividing line segments according to fixed proportions and the $S_j$ preserve proportions. We use this to show  inductively that  $A_{k}=\cup_{|w|=k} S_w\bigl([p_1,0]\cup [p_2,0]\bigr)$.  For $k=0$ this is trivial, as the edges of $A_0$ are precisely $[p_1,0]\cup [p_2,0]$, but our induction will actually start at the $k=1$ case.  To verify this latter, observe that $S_1$ fixes $p_1$ and scales $[p_1,0]$ by $N/(N+2)$, so it maps $[p_1,0]$ to $[p_1,2p_1/(N+2)]$ which is one edge of $A_1$. Similarly $S_2$ maps $[p_2,0]$ to $[p_2,2p_2/(N+2)]$, which is another edge of $A_1$.  By direct computation we then check that $S_1p_2=S_2p_1=(p_1+p_2)/(N+2)$, which is the midpoint of $[2p_1/(N+2),2p_2/(N+2)]$ and conclude that the union  $S_2([p_1,0])\cup S_1([p_2,0])$ is the third edge of $A_1$.

The inductive step uses that $A_{k-1}= \cup_{|w|=k-1} S_w\bigl([p_1,0]\cup [p_2,0]\bigr)$ to see that $A_{k-1}=S_1(A_{k-2})\cup S_{2}(A_{k-2})$, and therefore for $k\geq2$
\begin{equation*}
	A_k=D(A_{k-1}) 
	=D \bigl( S_1(A_{k-2}) \cup S_2(A_{k-2}) \bigr).
	\end{equation*}
All pairs of neighboring segments in $S_1(A_{k-2}) \cup S_2(A_{k-2})$ are internal to either $S_1(A_{k-2})$ or $S_2(A_{k-2})$, with the exception of the segments that meet at $S_1(p_2)=S_2(p_1)$.  However, by the inductive hypothesis these segments run from $S_1 S_2^{k-1}(0)$ to $S_1(p_2)=S_2(p_1)$ and from $S_2 S_1^{k-1}(0)$ to $S_1(p_2)=S_2(p_1)$.  Moreover there is $c$ so $S_2^{k-1}(0)=cp_2$ and $S_1^{k-1}(0)=cp_1$, because both $S_1$ and $S_2$ contract by the same factor along the respective directions $p_1$, and  $p_2$.  Using~\eqref{eq:Sj} and computing $(N-1)p_1\cdot p_2/\|p_1\|^2=-1$ we find
\begin{align*}
	S_1 S_2^{k-1}(0) = S_1 c p_2 = \frac{cp_2}{N+2} + \frac{2-c}{N+2}p_1\\
	S_2 S_1^{k-1}(0) = S_2 c p_1 = \frac{cp_1}{N+2} + \frac{2-c}{N+2}p_2
	\end{align*}
so that the midpoint is $(p_1+p_2)/(N+2)=S_1(p_2)=S_2(p_1)$.  This shows that the two ending segments of $S_1(A_{k-2})$ and $S_2(A_{k-2})$ form a single line segment in $A_{k-1}$ as soon as $k\geq2$.  It follows that all of the de~Rham construction for $D \bigl( S_1(A_{k-2}) \cup S_2(A_{k-2}) \bigr)$ occurs within either $S_1(A_{k-2})$ or $S_2(A_{k-2})$. Then commuting $D$ with the $S_j$, $j=1,2$ gives
\begin{equation*}
	A_k=D \bigl( S_1(A_{k-2}) \cup S_2(A_{k-2}) \bigr)
	=S_1 \bigl( D(A_{k-2}) \bigr) \cup S_2 \bigl( D(A_{k-2}) \bigr)
	= S_1 (A_{k-1}) \cup S_2 (A_{k-1}).
	\end{equation*}
This and the inductive hypothesis ensure $A_k=\cup_{|w|=k} S_w\bigl([p_1,0]\cup [p_2,0]\bigr)$, from which we deduce that  $A_k$ converges to a set $A$ that is invariant under the iterated function system $\{S_1,S_2\}$. Uniqueness of the attractor of the i.f.s.\ then ensures $A=\Gamma$.
\end{proof}

The de~Rham curves have been fairly extensively studied because they have applications in wavelets, approximation theory, and certain areas in computer science.  References for some of these may be found in a paper of Protasov~\cite{Prot04}, which makes a detailed study of the regularity of these curves.  It is proved there that the de~Rham curves are affine similar sets, which is the main point in the above proof of Theorem~\ref{thm:GammaisdeRham}; more importantly for the current work he proves (a more general version of) the following result.

\begin{thm}[\protect{\cite{Prot04}} Theorem~2]\label{thm:Protasov}
A de~Rham curve is $C^1$ if $r\in(0,\frac13]$. Moreover if $r\in(0,\frac14)$ the curve is not $C^2$ but has H\"{o}lder continuous derivative with H\"{o}lder exponent
\begin{equation*}
	\frac{\log \bigl( r(1-2r)\bigr)} { \log (r+\sqrt{4r-7r^2}) - \log2} - 2.
	\end{equation*} 
\end{thm}

\begin{proof}[Proof of Theorem~\protect{\ref{thm:Gammarectifiable}}]
From Lemma~\ref{lem:GammainP}, $\Gamma$ is a plane curve.  Theorem~\ref{thm:GammaisdeRham} shows that $\Gamma$ is a de~Rham curve, and by Theorem~\ref{thm:Protasov} it is $C^1$.
Now $x\in\Gamma$ is $T_w(X)$ for an infinite word $w$ with letters in $\{0,1\}$. If we write $[w]_m$ for the trunctation to the first $m$ letters then the sequence $T_{[w]_m}([p_1,p_2])$ consist of chords of $\Gamma$ that converge to $x$. Since $\Gamma$ is $C^1$ the normalized sequence $T_{[w]_m}T^t_{[w]_m}/\|T_{[w]_m}\|_{\text{HS}}^2$ converges to the operator of projection onto the  tangent direction to $\Gamma$ at $x$.
\end{proof}

Having established the basic regularity properties of $\Gamma$ our next goal is to show that it is the shortest path between $p_1$ and $p_2$ in $X$.  In order to proceed we collect some additional features of $\Gamma$ in the next result.

\begin{lem}\label{lem:factsaboutUpsilon}
$\Gamma$ and the line segment from $p_1$ to $p_2$ bound an open convex region. If $\Upsilon$ denotes the complement of this region in the triangle $\Omega$ with vertices $p,p_1,p_2$, then $\Upsilon$  is star-shaped with respect to $0$ and has the property that $(\Upsilon+ \alpha p)\cap \Omega\subset \Upsilon$ for any $\alpha\geq0$.
\end{lem}
\begin{proof}
It is useful to think of constructing the interior of the convex hulls of the approximating curves $A_k$ for our de~Rham curve $A$ as part of the inductive construction. If we denote the $k^{\text{th}}$ hull interior by $C_k$ one easily checks that $C_k$, $k\geq1$ can be obtained as follows: for two adjacent vertices $x,y\in A_k$ the line through $x$ and $y$ divides the plane into an open half-plane $K(x,y)$ that contains both $p_1$ and $p_2$ and its complementary (closed) half-plane $K'(x,y)$; $C_k$ is the intersection of $C_{k-1}$ with all $K(x,y)$ corresponding to adjacent vertices $x,y\in A_k$.
Evidently $\cap_k C_k$ is convex and has $\Gamma$ as a boundary arc; when it is further intersected with the half-plane that is bounded by the line through $p_1,p_2$ and contains $0$ we obtain the convex region asserted in the statement.  Then $\Upsilon$ is the intersection of $\Omega$ with the union of the $K'(x,y)$, with the latter  taken over all $k$ and all pairs of adjacent vertices from each $A_k$. 

Consider the above for three consecutive vertices, $x+v_1,x,x+v_2$ from $A_{k-1}$, $k>1$.  The convex set $\{x+\alpha_1v_1+\alpha_2v_2: \alpha_1>0, \alpha_2>0\}$ contains $p_1$ and $p_2$ (for which one of $\alpha_1$ or $\alpha_2$ is at least $1$) and is the intersection $K(x+v_1,x)\cap K(x,x+v_2)$. The vertices of $A_k$ introduced at the $k^{\text{th}}$ step are of the form $y_1=x+\beta_1v_1$, $y_2=x+\beta_2 v_2$ with $\beta_1,\beta_2\in(0,1/2)$, so the line through them has direction $\beta_1 v_1-\beta_2v_2$.  At any point $y$ on this line the cone $L(v_1,v_2,y)=\{y-\alpha_1 v_1-\alpha_2v_2: \alpha_1\geq0,\alpha_2\geq0\}$ is disjoint from the line (because it does not contain points of the form $y+\delta(\beta_1 v_1-\beta_2v_2)$), and cannot contain $p_1,p_2$ because these are reached from $x$ using non-negative multiples of $v_1$ and $v_2$, with at least one coefficient being $1$ or more. 

We make two observations from the reasoning in the preceding paragraph.   The first is that that $L(v_1,v_2,y)\subset K'(y_1,y_2)$ at every $y$ on the segment from $y_1$ to $y_2$.  Then, by induction, if $y\in A_{k}$ is between vertices $z,z'$ of $A_{k}$ we see $L(p_1,p_2,y)$ is contained in $K'(z,z')$ and hence $\Omega\cap L(p_1,p_2,y)\subset\Upsilon$. Taking the union over $k$ gives that $L(p_1,p_2,y)\subset\Upsilon$ if $y$ is on any $A_k$.  This immediately establishes the assertion about translation of $\Upsilon$ by $\alpha p$  for $\alpha>0$.  The second observation is that $L(v_1,v_2,z)\subset K'(y_1,y_2)$ for any $z\in K'(y_1,y_2)$, so in particular $L(v_1,v_2,x)\subset K'(y_1,y_2)$.  Induction from this shows that the first such cone, $L(p_1,p_2,0)$ is contained in $K(z,z')$ for each pair of adjacent vertices $z,z'$ in any $A_k$.  Since $0$ is in this cone we conclude that $0\in K(z,z')$ for each pair of adjacent vertices $z,z'$ in every $A_k$, and since $K(z,z')$ is convex and hence star-shaped with respect to any of its points we it follows that $\Upsilon$ is a union of regions star-shaped with respect to $0$ and is therefore itself star-shaped with respect to $0$.
\end{proof}

\begin{lem}\label{lem:nonconvexside}
$\Pi(X)\subset\Upsilon$.
\end{lem}
\begin{proof}
We show that the intersection of $\Pi^{-1}(\Upsilon)$ with the simplex having vertices $\{p_j\}_1^N$ is invariant under the iterated function system $\{S_j:j=1,\dotsc,N\}$.  It therefore must contain the attractor $X$.

Both $S_1$ and $S_2$ contract every vector that is orthogonal to $P$ by the same factor.   It follows that $\Pi\circ S_j\circ \Pi^{-1}$ is well-defined and equal to $S_j\bigr|_P$ for $j=1,2$.  Since $\Gamma$ is invariant under $S_1\bigr|_P$ and $S_2\bigr|_P$  it is easy to check that these maps take $\Upsilon$ into itself, whereupon $\Pi\circ S_j\circ \Pi^{-1}=S_j\bigr|_P$ for $j=1,2$ implies $\Pi^{-1}(\Upsilon)$ is invariant under $S_1$ and $S_2$.

By symmetry, to treat the maps $S_j$, $j\geq3$ it suffices to consider $S_3$.  If  one has $\sum_j \alpha_j=1$ then $x=\sum_j \alpha_j p_j$  is a point in the simplex and its projection to $P$ is $\Pi x=\alpha_1 p_1+\alpha_2 p_2+\sum_3^N \alpha_j p$.  We may also compute $S_3x$ using~\eqref{eq:Sj} and the fact that $(N-1)p_j\cdot p_k/\|p_j\|^2=-1$.  We obtain
\begin{gather*}
	S_3x= \frac{x}{N+2} + \frac{1}{N+2}\Bigl(2+ (N-1)\alpha_j - \sum_{k\neq j}\alpha_k \Bigr)p_j
		= \frac{x}{N+2} + \frac{N\alpha_j+1}{N+2}p_j\\
	\Pi S_3 x = \frac{\Pi x}{N+2} + \frac{N\alpha_j+1}{N+2}p
	\end{gather*}
From Lemma~\ref{lem:factsaboutUpsilon}, $\Upsilon$ is star-shaped with respect to $0$ and therefore $\Pi x\in\Upsilon$ implies $\Pi x/(N+2)\in\Upsilon$.  Lemma~\ref{lem:factsaboutUpsilon} further established that the image of $\Upsilon$ under translation by a positive multiple of $p$ followed by intersection with $\Omega$ is contained in $\Upsilon$. As $(N\alpha_j+1)(N+2)>0$ we conclude from this and the previous observation that $\Pi \circ S_3 x\in\Upsilon$ whenever $\Pi x\in\Upsilon$, or equivalently that $S_3$ (and thus, by symmetry, each $S_k$ for $k\geq3$) maps the intersection of  $\Pi^{-1}(\Upsilon)$ with the simplex having vertices $\{p_j\}$ to itself.
\end{proof}

We can now prove our main results. Let us write $Z_\ast= Z\circ \pi^{-1}\circ \Psi^{-1}$, which is  a well defined map from $X\setminus\Psi(V_\ast)$ to the set of infinite words $\Sigma$.
\begin{thm}\label{thm:mainthm}
Using the restriction of the Euclidean metric from $Y$, $\Gamma$ is a geodesic in $X$ connecting $p_1$ to $p_2$. There is a $C^1$ function $g:[0,1]\to\Gamma$ such that the length of $\Gamma$ is
\begin{equation*}
	l(\Gamma) = \int_0^1 \left\langle g'(t), Z_\ast(g(t))\, g'(t) \right\rangle^{1/2} \, dt.
	\end{equation*}
Moreover  we have
\begin{equation}\label{eqn:lengthbound}
	  \frac{1}{3}\leq   l(\Gamma) \leq 2
	\end{equation}
\end{thm}
\begin{proof}
$\Gamma$ is fixed by $S_1$ and $S_2$, so it is a subset of $X$.  The fact that that $\Gamma$ is rectifable ensures the infimum of the length of  curves in $X$ that contain   $p_1$ and $p_2$ is well-defined.  Moreover, if $\gamma$ is such a curve then $\Pi(\gamma)$ is rectifiable, contains $p_1$ and $p_2$ and has length $ l(\gamma)\geq l(\Pi(\gamma))$. By Lemma~\ref{lem:nonconvexside} $\Pi(\gamma)\subset\Upsilon$, so $\Pi(\gamma)\cup[p_1,p_2]$ surrounds  the convex set $T\setminus\Upsilon$. However it is a result of  classical geometry that then 
\begin{equation*}
	l\bigl(\Pi(\gamma)\cup[p_1,p_2]\bigr)\geq l\bigl(\Gamma\cup[p_1,p_2]\bigr)
\end{equation*}
so $\Gamma$ achieves the infimum and is a geodesic.

The same convexity argument shows~\eqref{eqn:lengthbound} because $\Omega\setminus\Upsilon$ is contained in the triangle with vertices $p_1,p_2,0$  and contains the triangle with vertices $p_1,p_2, S_1(p_2)=S_2(p_1)=(p_1+p_2)/(N+2)$. We need only check that the lengths of the sides are $|p_1|=|p_2|=\sqrt{(N-1)/2N}\leq 2$ and $|(N+1)p_1-p_2|/(N+2)=\sqrt{(N^2+N+2)/2(N+1)^2}\geq1/3$ because $N\geq 2$.

We may parametrize $\Gamma$ in the following way. Take $e_1= p_1-p_2 $ and $e_2=\bigl(N/(N-2)\bigr)^{1/2} (p_1+p_2)$ as a basis for the plane $P$; it is easily checked that these are orthonormal.  For $t\in[0,1]$ the fact that $\Gamma$ bounds a convex region ensures there is a unique point $g(t) = p_2+te_1 + se_2$ on $\Gamma$.  The function $g(t)$ is $C^1$ because $\Gamma$ is $C^1$, and $g'(t)$ is in the direction of $Z_\ast(g(t))$, so
 \begin{equation*}
	l(\Gamma) = \int_0^1 \|g't)\|\, dt = \int_0^1 \left\langle g'(t), Z_\ast(g(t))\, g'(t) \right\rangle^{1/2} \, dt. \qedhere
	\end{equation*}
\end{proof}

\begin{cor}\label{cor:mainresult}
If $p,q\in F_w(V_0)$ there is a geodesic from $p$ to $q$ with length comparable to the diameter of $F_w(X)$ which has a $C^1$ parametrization $g_{pq}$ for which~\eqref{eqn:geodesiceqn} gives the length of the geodesic.
\end{cor}
\begin{proof}
Given $1\leq j<k\leq N$ one can find an orthogonal transformation of $Y$ that maps $p_1$ to $p_j$ and $p_2$ to $p_k$.  Under this map the image of $\Gamma$ is a geodesic curve $\Gamma_{jk}$ from $p_j$ to $p_k$.  Evidently $\Gamma_{jk}$ is the invariant set of $\{T_j,T_k\}$.  Now there are $j,k$ such that $p=F_w(p_j)$ and $q=F_w(p_k)$  and $F_w(\Gamma_{jk})$ is a geodesic from $p$ to $q$ in $F_w(X)$.  The length bound follows from~\eqref{eqn:lengthbound} and the validity of~\eqref{eqn:geodesiceqn} comes from the fact that the restriction of $Z_\ast$ to $F_w(X)$ is $T_w Z_\ast$.
\end{proof}

The proof of Theorem~\ref{thm:geodesicspace} now follows the argument given to prove Theorem~5.1 on page~798 of~\cite{KigMsbleRGeom}, but is included below for the convenience of the reader.

\begin{proof}[Proof of Theorem~\protect{\ref{thm:geodesicspace}}]
Suppose $p,q\in V_\ast$, so there is $m$ so they are both in $V_m$.  There is a finite sequence $p=p_1,\dotsc,p_k=q$ with no repetitions such that $p_j$ and $p_{j+1}$ are in the boundary of $F_w(V_0)$ with $|w|=m$.   Concatenation of the geodesics from $p_j$ to $p_{j+1}$ that were constructed in Corollary~\ref{cor:mainresult} provides a finite length path from $p$ to $q$.

Conversely, any finite length path from $p$ to $q$ passes through some sequence of points from $V_m$. We may shorten the path by deleting loops, at which point this sequence $p=p_1,\dotsc,p_k=q$  is finite and has no repetitions. The curve constructed above from geodesics joining $p_j$ to $p_{j+1}$ is evidently the shortest path from $p$ to $q$ that passes through these points in this order.  Moreover there are finitely many sequences of this type, so by taking the minimum of the lengths of the resulting finite collection of paths  we find a geodesic from $p$ to $q$ which we denote $\gamma_{pq}$.  This proves the result when both points are in $V_\ast$. 

Observe that in the above construction if $p,q\in V_m$ and  $\gamma_{pq}$ passes through $p',q'\in V_{m'}$ for some $m'<m$ then replacing the arc of $\gamma_{pq}$ from $p'$ to $q'$ with $\gamma_{p'q'}$ does not increase the length.  Thus for sequences $p_j\to q$ and $q_j\to q$  there are geodesics $\gamma_{j}$  from $p_j$ to $q_j$ such that $\gamma_j$ is an arc of $\gamma_k$ for each $j<k$.  Then $\cup_j\gamma_j$ is  a geodesic from $p$ to $q$ and is made up of arcs as constructed in Corollary~\ref{cor:mainresult}.

 To each of the above arcs there is a finite word and an orthogonal transformation as in Corollary~\ref{cor:mainresult} so that the composition of $F_w$, the orthogonal map and the function $g$ from Theorem~\ref{thm:mainthm}  is a $C^1$ parametrization of the arc. Concatenating these functions (and arcs) gives a parametrization of the geodesic from $p$ to $q$, and it satisfies~\eqref{eqn:geodesiceqn} because Theorem~\ref{thm:mainthm} shows the parametrizations of individual arcs had this property.  Moreover it is $C^1$ because at any point of $V_\ast$ where two arcs from the construction are joined, the one-sided tangents are equal in the same way that they were when we joined the two halves of $\Gamma$ in the proof of Theorem~\ref{thm:GammaisdeRham}. 
\end{proof}

\section{Acknowledgements}
The authors thank Alexander Teplyaev for useful discussions and suggestions.

\end{document}